\documentclass[12pt]{article}
\usepackage{amsmath,amsthm,amssymb}
\usepackage{mathrsfs}
\usepackage{color}
\DeclareMathOperator{\Irr}{Irr}

\DeclareMathOperator{\tr}{tr}

\begin{document}
\input amssym.def
\setcounter{equation}{0}
\newcommand{\wt}{{\rm wt}}
\newcommand{\spa}{\mbox{span}}
\newcommand{\Res}{\mbox{Res}}
\newcommand{\SL}{\mbox{SL}}
\newcommand{\End}{\mbox{End}}
\newcommand{\Ind}{\mbox{Ind}}
\newcommand{\Hom}{\mbox{Hom}}
\newcommand{\Mod}{\mbox{Mod}}
\newcommand{\m}{\mbox{mod}\ }
\renewcommand{\theequation}{\thesection.\arabic{equation}}
\numberwithin{equation}{section}

\def \End{{\rm End}}
\def \stab{{\rm stab}}
\def \Aut{{\rm Aut}}
\def \Z{\mathbb Z}
\def \H{\mathbb H}
\def \M{\Bbb M}
\def \C{\mathbb C}
\def \R{\mathbb R}
\def \Q{\mathbb Q}
\def \N{\mathbb N}
\def \ann{{\rm Ann}}
\def \<{\langle}
\def \o{\omega}
\def \O{\Omega}
\def \Or{\cal O}
\def \M{{\cal M}}
\def \1t{\frac{1}{T}}
\def \>{\rangle}
\def \t{\tau }
\def \a{\alpha }
\def \e{\epsilon }
\def \l{\lambda }
\def \L{\Lambda }
\def \g{\gamma}
\def \b{\beta }
\def \om{\omega }
\def \o{\omega }
\def \ot{\otimes}
\def \cg{\chi_g}
\def \ag{\alpha_g}
\def \ah{\alpha_h}
\def \ph{\psi_h}
\def \S{\cal S}
\def \nor{\vartriangleleft}
\def \V{V^{\natural}}
\def \voa{vertex operator algebra\ }
\def \voas{vertex operator algebras}
\def \v{vertex operator algebra\ }
\def \1{{\bf 1}}
\def \be{\begin{equation}\label}
\def \ee{\end{equation}}
\def \qed{\mbox{ $\square$}}
\def \pf {\noindent {\bf Proof:} \,}
\def \bl{\begin{lem}\label}
\def \el{\end{lem}}
\def \ba{\begin{array}}
\def \ea{\end{array}}
\def \bt{\begin{thm}\label}
\def \et{\end{thm}}
\def \br{\begin{rem}\label}
\def \er{\end{rem}}
\def \ed{\end{de}}
\def \bp{\begin{prop}\label}
\def \ep{\end{prop}}
\def \p{\varphi}
\def \d{\delta}
\def \irr{\rm irr}

\newtheorem{th1}{Theorem}
\newtheorem{ree}[th1]{Remark}
\newtheorem{thm}{Theorem}[section]
\newtheorem{prop}[thm]{Proposition}
\newtheorem{coro}[thm]{Corollary}
\newtheorem{lem}[thm]{Lemma}
\newtheorem{rem}[thm]{Remark}
\newtheorem{de}[thm]{Definition}
\newtheorem{hy}[thm]{Hypothesis}
\newtheorem{conj}[thm]{Conjecture}
\newtheorem{ex}[thm]{Example}

\newcommand\red[1]{{\color{red} #1}}
\begin{center}
{\Large {\bf $S$-matrix in orbifold theory}}\\
\vspace{0.5cm}
Chongying Dong\footnote
{Partially supported by   NSFC 11871351 }\\
\ Department of Mathematics, University of
California, Santa Cruz, CA 95064 USA 
\\ 
Li Ren\footnote{Supported by NSFC  11671277 }\\
 School of Mathematics,  Sichuan University,
Chengdu 610064 China\\
Feng Xu\footnote{Partially supported by  NSFC 11871150}\\
Department of Mathematics, University of California, Riverside, CA
92521, USA 
\end{center}

\begin{abstract}
The restricted $S$-matrix of $V^G$ is determined for any regular vertex operator algebra $V$ and finite automorphism group $G$ of $V.$  As an application,
 the $S$-matrices for cyclic permutation orbifolds of prime orders are computed.
\end{abstract}

\section{Introduction}
The orbifold theory studies a vertex operator algebra $V$ under the action of a finite automorphism group $G.$ It is proved recently in \cite{M, CM} that if $V$ is regular and $G$ is solvable then $V^G$ is also regular. Under the assumption that $V^G$ is regular, every  irreducible $V^G$-module occurs in an irreducible $g$-twisted $V$-module \cite{DRX}.
In particular, this result holds for $G$ being solvable.  If $V$ is rational and  $V^G$ $C_2$-cofinite, then $V^G$ is also rational \cite{Mc}. In this paper, we study the $S$-matrix
for $V^G$ under the assumption that $V$is regular. 

The modularity of trace functions in the theory of rational vertex operator algebra plays a power role in studying vertex operator algebra. The conformal block of a regular
vertex operator algebra spanned by the trace functions on the irreducible $V$-modules affords to a representation of the modular group $\Gamma=SL_2(\Z)$ \cite{Z}. Moreover,
the kernel of this representation is a congruence subgroup \cite{DLN}. This implies that the trace functions in rational vertex operator algebras are modular forms on congruence subgroup. In particular, the irreducible characters are modular functions. The results in \cite{Z} are extended to  include an action of finite groups in \cite{DLM5} to get a modular
invariance of the trace functions in orbifold theory. In the case that  $V$ is  holomorphic this result provides a framework for the generalized moonshine \cite{N}. Also, the trace functions in orbifold theory are modular forms on a congruence subgroup without the assumption that $V^G$ is rational \cite{DR}.

It is well known that $\Gamma$ is generated by $S=\left(\begin{array}{cc} 0 &-1\\ 1 &0\end{array}\right)$ and $T=\left(\begin{array}{cc} 1 &1\\ 1 &0\end{array}\right).$ The action of $T$
on conformal block with respect to a distinguished basis is a diagonal matrix determined by the central charge and the conformal weights of irreducible modules. The action of $S$ is called the $S$-matrix which  is the key to understand the
action of $\Gamma.$ The celebrated Verlinde formula \cite{V,H} exhibits  the beauty and importance of the $S$-matrix.  $S$-matrix has also been used to determine the number of irreducible $g$-twisted modules \cite{DLM5} and classification of irreducible modules.

The so-called restricted $S$-matrix of $V^G$ is the restriction of the $S$-matrix of $V^G$ to the irreducible $V^G$-modules appearing in the twisted modules. There is no doubt that
this should be the full $S$-matrix once the $C_2$-cofiniteness of $V^G$ is established. More explicitly, we give a precise formula for the restricted $S$-matrix in terms of the $S$-matrix of trace functions of the twisted modules obtained in \cite{DLM5}. The main idea is that the space of the conformal block of $V^G$ spanned by the trace functions on the
irreducible $V^G$-modules appearing in twisted $V$-modules in \cite{Z} is equal to  the twisted conformal block of $V$ spanned by the trace functions on the irreducible twisted modules in \cite{DLM5}. Some entries of the restricted $S$-matrix have been computed previously in \cite{DRX}  for studying the quantum dimensions and global dimensions
for vertex operator algebras $V^G$ \cite{DJX}. In the case $G$ is abelian or $V$ is holomorphic, the $S$-matrices have simpler expressions. The $S$-matrix for cyclic group $G$ and holomorphic vertex operator algebra has been investigated in \cite{EMS} for the purpose of constructing holomorphic vertex operator algebras with central charge $24.$

As an application, we determine the $S$-matrix for the cyclic permutation orbifold of prime order. Let $k$ be a positive integer. Then $V^{\otimes k}$ is a vertex operator algebra \cite{FHL} and
$S_k$ acts on $V^{\otimes k}$ as automorphisms. For any $g\in S_k,$ the $g$-twisted $V^{\otimes k}$-module category was determined in \cite{BDM} in terms of $V$-module category.
In particular, if $g$ is a cycle, there is an equivalence between $V$-module category and $g$-twisted $V^{\otimes k}$-module category. For general orbifold theory, one do not know how to construct twisted modules although the number of irreducible twisted modules are known.
But for the permutation orbifold theory one has an explicit construction of twisted modules in terms of $V$-
modules \cite{BDM}.
In this paper when $k$ is a prime and $g$ is a $k$-cycle and $G$ is generated by $g$ we give an explicit formula of $S$-matrix of $(V^{\otimes k})^G$ in terms of the action of $\Gamma $ on the conformal block of $V.$ The assumption that $k$ is a prime ensures that every power of $g$ is a cycle, so the $S$-matrix of  $(V^{\otimes k})^G$  has a simple formula. One can consider general case with more complicated computation.
The $S$-matrix of the permutation orbifold for
any subgroup of $S_k$ has been studied in \cite{B} and \cite{KLX} from the point views of conformal field theory and conformal nets.

For the general orbifold theory, one can see \cite{DLM3, DLM4} for the study of  twisted modules, \cite{DLM1, DM, MT,
HMT, DY, DJX} for  Schur-Weyl duality and quantum Galois theory in orbifold theory, \cite{X} for conformal net approach to the orbifold theory  and \cite{DPR} for the connection between the holomorphic orbifold theory and the twisted Drinfeld
double \cite{D}.  One can also find the construction of twisted modules in \cite{FLM1, FLM2, L, DL, Li1} in terms of $\Delta$-operators.

The paper is organized as follows. In Section 2, we review twisted modules, $g$-rationality and related results from \cite{DLM3}. The modular invariance result on trace functions for
the orbifold theory from \cite{DLM5} is given in Section 3. Classification of the irreducible $V^G$-modules occurring in irreducible twisted modules from \cite{DLM1, DY, MT, DRX}
is present in Section 4.
Section 5 is devoted to the study of the restricted $S$-matrix for any regular vertex
operator algebra $V$ and any finite automorphism group $G.$
In particular, we give an explicit formula of the restricted $S$-matrix of $V^G$ in terms of the $S$-matrix from \cite{DLM5} for the action of $\Gamma$ on the twisted conformal block
of $V$.
Section 6 is an application of Section 5 to the cyclic permutation orbifolds.

\section{Basics}
 Various notions of twisted modules for a \voa following \cite{DLM3} are reviewed in this section. The concepts such as  rationality, regularity, and $C_2$-cofiniteness from \cite{Z} and \cite{DLM2} are discussed.

Let $V$ be a vertex operator algebra and $g$ an automorphism of $V$ of finite order $T$. Then  $V$ is a direct sum of eigenspaces of $g:$
\begin{equation*}\label{g2.1}
V=\bigoplus_{r\in \Z/T\Z}V^r,
\end{equation*}
where $V^r=\{v\in V|gv=e^{-2\pi ir/T}v\}$.
We use $r$ to denote both
an integer between $0$ and $T-1$ and its residue class \m $T$ in this
situation.

\begin{de} \label{weak}
A {\em weak $g$-twisted $V$-module} $M$ is a vector space equipped
with a linear map
\begin{equation*}
\begin{split}
Y_M: V&\to (\End\,M)[[z^{1/T},z^{-1/T}]]\\
v&\mapsto\displaystyle{ Y_M(v,z)=\sum_{n\in\frac{1}{T}\Z}v_nz^{-n-1}\ \ \ (v_n\in
\End\,M)},
\end{split}
\end{equation*}
which satisfies the following:  for all $0\leq r\leq T-1,$ $u\in V^r$, $v\in V,$
$w\in M$,
\begin{eqnarray*}
& &Y_M(u,z)=\sum_{n\in \frac{r}{T}+\Z}u_nz^{-n-1} \label{1/2},\\
& &u_lw=0~~~
\mbox{for}~~~ l\gg 0,\label{vlw0}\\
& &Y_M({\mathbf 1},z)=Id_M,\label{vacuum}
\end{eqnarray*}
 \begin{equation*}\label{jacobi}
\begin{array}{c}
\displaystyle{z^{-1}_0\delta\left(\frac{z_1-z_2}{z_0}\right)
Y_M(u,z_1)Y_M(v,z_2)-z^{-1}_0\delta\left(\frac{z_2-z_1}{-z_0}\right)
Y_M(v,z_2)Y_M(u,z_1)}\\
\displaystyle{=z_2^{-1}\left(\frac{z_1-z_0}{z_2}\right)^{-r/T}
\delta\left(\frac{z_1-z_0}{z_2}\right)
Y_M(Y(u,z_0)v,z_2)},
\end{array}
\end{equation*}
where $\delta(z)=\sum_{n\in\Z}z^n$ and
all binomial expressions (here and below) are to be expanded in nonnegative
integral powers of the second variable.
\end{de}

\begin{de}\label{ordinary}
A $g$-{\em twisted $V$-module} is
a $\C$-graded weak $g$-twisted $V$-module $M:$
\begin{equation*}
M=\bigoplus_{\lambda \in{\C}}M_{\lambda}
\end{equation*}
such that
$\dim M_{\l}$ is finite and for fixed $\l,$ $M_{\frac{n}{T}+\l}=0$
for all small enough integers $n,$ where $M_{\l}=\{w\in M|L(0)w=\l w\}$ and $L(0)$ is the component operator of $Y(\omega,z)=\sum_{n\in \Z}L(n)z^{-n-2}.$    If $w\in M_{\l}$ we refer to $\l$ as the {\em weight} of
$w$ and write $\l=\wt w.$
\end{de}

Let $\Z_+$ be the set of nonnegative integers.
\begin{de}\label{admissible}
 An {\em admissible} $g$-twisted $V$-module
is a  $\frac1T{\Z}_{+}$-graded weak $g$-twisted $V$-module $M:$
\begin{equation*}
M=\bigoplus_{n\in\frac{1}{T}\Z_+}M(n)
\end{equation*}
such that
\begin{equation*}
v_mM(n)\subseteq M(n+\wt v-m-1)
\end{equation*}
for homogeneous $v\in V,$ $m,n\in \frac{1}{T}{\Z}.$
\ed

If $g=Id_V$  we have the notions of  weak, ordinary and admissible $V$-modules \cite{DLM2}.

\begin{de}
A \voa $V$ is called $g$-rational, if the  admissible $g$-twisted module category is semisimple. $V$ is called rational if $V$ is $1$-rational.
\end{de}

There is another important concept called $C_2$-cofiniteness \cite{Z}.
\begin{de}
We say that a \voa $V$ is $C_2$-cofinite if $V/C_2(V)$ is finite dimensional, where $C_2(V)=\langle v_{-2}u|v,u\in V\rangle.$
\end{de}

The following results about $g$-rational \voas \  are well-known \cite{DLM3,DLM5}.
\begin{thm}\label{grational}
If $V$ is $g$-rational, then

(1) Any irreducible admissible $g$-twisted $V$-module $M$ is a $g$-twisted $V$-module. Moreover, there exists a number $\l \in \mathbb{C}$ such that  $M=\oplus_{n\in \frac{1}{T}\mathbb{Z_+}}M_{\l +n}$ where $M_{\lambda}\neq 0.$ The $\l$ is called the conformal weight of $M;$

(2) There are only finitely many irreducible admissible  $g$-twisted $V$-modules up to isomorphism.

(3) If $V$ is also $C_2$-cofinite and $g^i$-rational for all $i\geq 0$ then the central charge $c$ and the conformal weight $\l$ of any irreducible $g$-twisted $V$-module $M$ are rational numbers.
\end{thm}

\begin{de}
A \voa $V$ is called regular if every weak $V$-module is a direct sum of irreducible $V$-modules.
\end{de}

A \voa $V=\oplus_{n\in \Z}V_n$  is said to be of CFT type if $V_n=0$ for negative
$n$ and $V_0=\C {\bf 1}.$
It is proved in \cite{Li2} and \cite{ABD} that if  $V$ is of CFT type, then regularity is equivalent to rationality and $C_2$-cofiniteness. Also $V$ is regular if and only if the weak module category is semisimple \cite{DYu}.

In the rest of this  paper we assume the following:
\begin{enumerate}
\item[(V1)] $V=\oplus_{n\geq 0}V_n$ is a simple, rational
\red{,}
$C_2$-cofinite  vertex operator algebra  of CFT type,
\item[(V2)] $G$ is a finite automorphism group of $V,$
\item[(V3)] The conformal weight of any irreducible $g$-twisted $V$-module $M$
 is nonnegative and is zero if and only if $M=V.$
\end{enumerate}

We remark that with  these assumptions, $V^G$ is $C_2$-cofinite, rational if $G$ is solvable \cite{M,CM}. Moreover, $V$ is $g$-rational for any finite order automorphism $g$ \cite{ADJR,DRX}.

\section{Modular Invariance}

This section largely follows from \cite{DLM5}.  We present the main results on modular invariance in orbifold theory.

For the modular invariance, we need the action of $\Aut(V)$ on twisted modules \cite{DLM5}. Let $g, h$ be two automorphisms of $V$ with $g$ of finite order. If $(M, Y_M)$ is a weak $g$-twisted $V$-module, there is a weak $h^{-1}gh$-twisted  $V$-module $(M\circ h, Y_{M\circ h})$ where $M\circ h\cong M$ as vector spaces and
$Y_{M\circ h}(v,z)=Y_M(hv,z)$ for $v\in V.$
This defines a right  action of $\Aut(V)$ on weak twisted $V$-modules and on isomorphism
classes of weak twisted $V$-modules. Symbolically, we write
\begin{equation*}
(M,Y_M)\circ h=(M\circ h,Y_{M\circ h})= M\circ h.
\end{equation*}
$M$ is called $h$-stable if $M, M\circ h$ are isomorphic.
It is proved in \cite{DLM5} that if $M$ is an admissible $g$-twisted $V$-module, then $M\circ g$ and $M$ are isomorphic.

Assume that  $g,h$ commute. Then  $h$  acts on the $g$-twisted modules.
 Denote by $\mathscr{M}(g)$ the equivalence classes of irreducible $g$-twisted $V$-modules and set $\mathscr{M}(g,h)=\{M \in \mathscr{M}(g)| M\circ h\cong M\}.$
 Since  $V$ is $g$-rational for all $g$, both $\mathscr{M}(g)$ and $\mathscr{M}(g,h)$ are finite sets.
 The following lemma is obvious from the definitions. It will be useful later.

 \begin{lem}\label{circ}
 Let $g,h\in G$, $M$ an irreducible $g$-twisted $V$-module.
 For any $k\in G$, we have

(1) $\mathscr{M}(g,h)$ and $\mathscr{M}(k^{-1}gk, k^{-1}hk)$ have the same cardinality.

(2) $M\circ k\in \mathscr{M}(g)$ if and only if $k\in C_G(g)$.

(3) $M$ and $M\circ h$ are isomorphic $V^G$-modules.
\end{lem}

 Let $M$ be an irreducible $g$-twisted $V$-module and $G_M$ is a subgroup of $G$ consisting of $h\in G$ such that $M\circ h$ and $M$ are isomorphic.  Using the Schur's Lemma
 gives a projective representation $\varphi$ of $G_M$ on $M$ such that
 $$\varphi(h)Y(u,z)\varphi(h)^{-1}=Y(hu,z)$$
  for $h\in G_M.$  If $h=1$ we simply take $\varphi(1)=1.$
For homogeneous $v\in V$ we  set $o(v)=v_{\wt v-1}$ which is the degree zero operator of $v.$
Here and below $\tau$ is in the complex upper half-plane $\H$ and $q=e^{2\pi i\tau}.$
For $v\in V$ we set
\begin{equation}\label{e1.00}
Z_M(v, (g,h),\tau)=\tr_{_M}o(v)\varphi(h) q^{L(0)-c/24}=q^{\lambda-c/24}\sum_{n\in\frac{1}{T}\Z_+}\tr_{_{M_{\l+n}}}o(v)\varphi(h)q^{n}
\end{equation}
which is a holomorphic function on the $\H$ \cite{DLM5,Z}.
Note that $Z_M(v, (g,h),\tau)$ is defined up to a nonzero scalar. For short we set  $Z_M(v,\tau)=Z_M(v,(g,1),\tau).$  Then $\chi_M(\tau)=Z_M(\1,\tau)$   is called the character
of $M.$

There is another vertex operator algebra $(V, Y[~], \1, \tilde{\omega})$ associated to $V$
in \cite{Z}.  Here $\tilde{\omega}=\omega-c/24$ and
$$Y[v,z]=Y(v,e^z-1)e^{z\cdot \wt v}=\sum_{n\in \Z}v[n]z^{n-1}$$
for homogeneous $v.$ We also write
$$Y[\tilde{\omega},z]=\sum_{n\in \Z}L[n]z^{-n-2}.$$
If $v\in V$ is homogeneous in the second vertex operator algebra, we denote its weight by $\wt [v].$

Let $P(G)$ denote the ordered commuting pairs in $G.$
For $(g,h)\in P(G)$ and $M\in \mathscr{M}(g,h),$ $Z_M(v,(g,h),\tau)$ is a function
on $V\times \H.$ Let $W$ be the vector space spanned by these functions. It is clear that
the dimension of $W$ is equal to $\sum_{(g,h)\in P(G)}|\mathscr{M}(g,h)|$ \cite{DLM5}.
We now define an action of the modular group $\Gamma$ on $W$ such that
\begin{equation*}
Z_M|_\gamma(v,(g,h),\tau)=(c\tau+d)^{-{\rm wt}[v]}Z_M(v,(g,h),\gamma \tau),
\end{equation*}
where
\begin{equation}\label{e1.11}
\gamma: \tau\mapsto\frac{ a\tau + b}{c\tau+d},\ \ \ \gamma=\left(\begin{array}{cc}a & b\\ c & d\end{array}\right)\in\Gamma=SL(2,\Z).
\end{equation}
 We let $\gamma\in \Gamma$ act on the right of $P(G)$ via
$$(g, h)\gamma = (g^ah^c, g^bh^d ).$$

The following results are established in \cite{DLM5,Z,DLN,DR}.
\begin{thm}\label{minvariance} Let $V$ and $G$ be as before. Then

(1) There is a representation $\rho: \Gamma\to GL(W)$
 such that for $(g, h)\in P(G),$   $\gamma =\left(\begin{array}{cc}a & b\\ c & d\end{array}\right)\in \Gamma,$
and $M\in \mathscr{M}(g,h),$
$$
Z_{M}|_{\gamma}(v,(g,h),\tau)=\sum_{N\in \mathscr{M}(g^ah^c,{g^bh^d)}} \gamma_{M,N}Z_{N}(v,(g,h)\gamma, ~\tau)$$
where $\rho(\gamma)=(\gamma_{M,N}).$
That is,
$$Z_{M}(v,(g,h),\gamma\tau)=(c\tau+d)^{{\rm wt}[v]}\sum_{N\in \mathscr{M}(g^ah^c,{g^bh^d)}} \gamma_{M,N}Z_{N}(v,(g^ah^c, g^bh^d), ~\tau).$$

(2) The cardinalities $|\mathscr{M}(g,h)|$ and $|\mathscr{M}(g^ah^c,g^bh^d)|$ are equal for any $(g,h)\in P(V)$ and
$\gamma\in \Gamma.$ In particular, the number of irreducible $g$-twisted $V$-modules is exactly the number of irreducible $V$-modules which are  $g$-stable.

(3) The kernal of $\rho$ is a congruence subgroup of $SL_2(\Z).$ That is, each $
Z_{M}|_{\gamma}(v,(g,h),\tau)$ is a modular form of weight $\wt[v]$ on the congruence subgroup. In particular, the character $\chi_M(\tau)$ is a modular function on the same congruence subgroup.
\end{thm}

The modular transformation formula in (1) with $G=\{1\}$ was obtained in \cite{Z}.
(1), (2) were given in \cite{DLM5} for general $G.$ (3)  is a result from \cite{DLN} with $G=\{1\}$ and is proved in \cite{DR} for general $G.$

It is well known that the modular group $\Gamma$ is generated by $S=\left(\begin{array}{cc}0 & -1\\ 1 & 0\end{array}\right)$ and $T=\left(\begin{array}{cc}1 & 1\\ 0 & 1\end{array}\right).$  So the representation $\rho$ is uniquely determined by $\rho(S)$ and  $\rho(T).$ It is almost trivial to compute $\rho(T)$ once we know the  irreducible modules.  The matrix $\rho(S)$
is called the $S$-matrix of the orbifold theory.  Here is a special case of the $S$-transformation:
\begin{equation}\label{S-tran1}
Z_{M}(v,-\frac{1}{\tau})=\tau^{\wt[v]}\sum_{N\in \mathscr{M}(1,g^{-1}) }S_{M,N}Z_{N}(v,(1,g^{-1}),\tau)
\end{equation}
for $M\in \mathscr{M}(g)$ and
\begin{equation}\label{S-tran2}
Z_{N}(v,(1,g),-\frac{1}{\tau})=\tau^{\wt[v]}\sum_{M\in \mathscr{M}(g)} S_{N,M}Z_{M}(v,\tau)
\end{equation}
for $N\in \mathscr{M}(1).$  We will use $\mathscr{M}_V$ for $\mathscr{M}(1).$
The matrix $S=(S_{M,N})_{M,N\in \mathscr{M}_V}$ is called $S$-matrix of $V$. We need the following result from \cite{DLN}  later.
\begin{prop} \label{pDLN}The $S$-matrix is unitary and $S_{V,M}=S_{M,V}$ is positive for any irreducible $V$-module $M.$
\end{prop}

\section{Irreducible modules for $V^G$}

In this section we give the irreducible $V^G$-modules appearing in an irreducible $g$-twisted $V$-module for some $g\in G$ \cite{DRX}.
If $G$ is solvable,  these are the all irreducible $V^G$-modules. For general $G$, this is also true if $V^G$ is rational and $C_2$-cofinite.

Let $M=(M,Y_M)$ be an irreducible $g$-twisted $V$-module. We have discussed that  $G_M$ acts on $M$ projectively.  Let $\a_M$  be the corresponding 2-cocycle in $C^2(G,\C^{\times}).$ Then $\varphi(h)\varphi(k)=\a_M(h,k)\varphi(hk)$
for all $h,k\in G_M.$
We may  assume that $\alpha_M$ is unitary  in the sense that there is a fixed positive integer $n$ such that  $\alpha_M(h,k)^n=1$ for all $h,k\in G_M.$
Let $\C^{\a_M}[G_M]=\oplus_{h\in G_M} \C\bar h$ be the twisted group algebra with product $\bar h\bar k=\alpha_M(h,k)\bar{hk}.$ It is well known that $\C^{\a_M}[G_M]$ is a semisimple associative algebra.
It follows that $M$ is a module for
$\C^{\a_M}[G_M].$
Note that $G_M$ is a subgroup of $C_G(g)$ and  $g$ lies in $G_M.$  Let $M=\oplus_{n\in \frac{1}{T}\Z_+}M(n)$ and $M(0)\ne 0.$ Then $\varphi(g)$ acts on $M(n)$ as $e^{2\pi in}$ for all $n$  \cite{DRX}.

 Let $\Lambda_{G_M}$ be the set of all irreducible characters $\lambda$
of  $\C^{\a_M}[G_M]$. Denote the corresponding simple module by $W_{\l}.$  Let $M^{\lambda}$ be the sum of simple
$\C^{\a_M}[G_M]$-submodules of $M$ isomorphic
to $W_{\l}.$ Then
\begin{equation}\label{decom}
M=\oplus_{\lambda\in \Lambda_{G_M}}M^{\lambda}=\oplus_{\lambda\in \Lambda_{G_M}}W_{\l}\otimes M_{\lambda}
\end{equation}
where the multiplicity space $M_{\lambda}$ of $W_{\l}$ in $M$ is a  $V^{G}$-module.

Recall that the group $G$ acts on set ${\cal S}=\cup_{g\in G}\mathscr{M}(g)$ and $ M\circ h$ and $M$ are isomorphic
$V^G$-modules for any $h\in G$ and $M\in {\cal S}.$ It is clear that the cardinality of the $G$-orbit $|M\circ G|$ of $M$ is
$[G:G_M].$
The following results are obtained in  \cite{DRX}, and \cite{MT} (also see \cite{DLM1}, \cite{DY}) .

 \begin{thm}\label{mthm1}  Let $g,h\in G,$ $M$ an irreducible $g$-twisted $V$-module, $N$ an irreducible $h$-twisted $V$-module.
1) $M^{\l}$ is nonzero for any $\l\in \Lambda_{G_M}.$

2) Each $M_{\l}$ is an irreducible  $V^{G}$-module.

3) $M_{\l}$ and $M_{\gamma}$ are equivalent
 $V^{G}$-module if and only
if $\l=\gamma.$

4) For any $\lambda\in \Lambda_{G_M}$ and $\mu \in \Lambda_{G_N},$ the irreducible $V^G$-modules
$M_{\l}$ and $N_{\mu}$ are inequivalent if $M,N$ are not in the same orbit of ${\cal S}$ under the action of $G.$
 \end{thm}

Decompose  $\cal S$ $=\cup_{ j \in J}O_{j}$ into a disjoint union of orbits.
Let $M^j$  for $j\in J$ be the orbit representatives of  $\cal S$,
$O_{j}=\{M^j\circ h\mid h\in G\}$ is  the orbit of $M^j$ under $G$.
For $M\in  \cal S$, recall that $G_M=\{h\in G\mid M\circ h=M\}$.
Let $G=\cup_{i=1}^{l_M} G_M g_i$ be a right coset decomposition of $G$. Assume $g_1=1$ is the identity.
Then $O_M=\{M\circ g_i \mid i=1,\ldots, l_M \}$ and $G_{M\circ g_i}={g_i}^{-1}G_M g_i$. Then we have

 \begin{thm}\label{mthm1}
The set $\{M^j_{\lambda}|j\in J, \lambda\in
\Lambda_{G_{M^j}}\} $ gives a complete list of inequivalent irreducible $V^G$-modules appearing in the irreducible twisted $V$-modules.
Moreover, if $V^G$ is rational and $C_2$-cofinite, this set classifies the irreducible $V^G$-modules.  That is, any irreducible $V^G$-module is isomorphic to an irreducible $V^G$-submodule $M_{\l}$ for some irreducible $g$-twisted $V$-module $M$ and some $\lambda\in \Lambda_{G_M}.$
\end{thm}

The following Lemma which is straightforward will be useful in the next section.
\begin{lem}\label{GM}  Let $M$ be an irreducible $g$-twisted $V$-module. Then for any $k\in G$ we may take $\alpha_{M\circ k}$ such that $\alpha_{M\circ k}(k^{-1}ak,k^{-1}bk)=\alpha_{M}(a,b)$
for $a,b\in G_M.$ In particular, $\C^{\a_{M\circ k}}[G_{M\circ k}]$ and $\C^{\a_M}[G_M]$ are isomorphic by sending $\bar a$ to $\overline{k^{-1}ak}.$
\end{lem}

For any $\lambda\in \Lambda_{G_M}$ we define  $\lambda\circ k\in  \Lambda_{G_{M\circ k}}$ such that $ (\lambda\circ k)(\overline{k^{-1}ak})=
\lambda(\bar{a})$  for  $a\in G_M.$
We denote the irreducible $\C^{\a_{M\circ k}}[G_{M\circ k}]$-module corresponding to $\lambda\circ k$ by $W_{\lambda}\circ k=W_{\lambda\circ k}.$

\section{$S$-matrix}

In this section we investigate the modular transformation among the modular forms $Z_{M^i_\l}(v,\tau)$ for $i\in J$ and $\lambda\in \Lambda_{M^i}$  without the
assumption that $V^G$ is rational and $C_2$-cofinite.  Note that  the modular group
$\Gamma$ is generated by $S$ and $T.$ The $T$-matrix is easy to compute:
$$Z_{M^i_{\l}}(v,\tau+1)=e^{2\pi i(-c/24+h_{i,\l})}Z_{M_{\l}}(v,\tau)$$
where $h_{i,\l}$ is a rational number determined by the weight space decomposition
$$M^i_\l=\bigoplus_{n\geq 0}(M^i_\l)_{h_{i,\l}+n}.$$
That is $T$ is a diagonal unitary matrix with $T_{M^i_{\l},M^i_{\l}}=e^{2\pi i(-c/24+h_{i,\l})}$ for   $i\in J$ and $\l\in\Lambda_{M^i}.$

So our main focus is to show that any
$\tau^{-\wt[v]}Z_{M^i_\l}(v,-\frac{1}{\tau})$
is a linear combination of $Z_{M^j_\mu}(v,\tau)$ for $j\in J$ and $\mu\in\Lambda_{M^j}.$ That is, we will determine the $S_{M^i_\l,M^j_{\mu}}$ in
$$Z_{M^i_\l}(v,-\frac{1}{\tau}) = \tau^{\wt[v]}\sum_{j\in J, \mu\in\Lambda_{M^j}}S_{M^i_\l,M^j_{\mu}}Z_{M^j_\mu}(v,\tau).$$
$(S_{M^i_\l,M^j_{\mu}})_{(i,\lambda), (j,\mu)}$ is called the restricted $S$-matrix of $V^G.$

We need the following lemma which says that the $S$-matrix is invariant under the conjugation.
\begin{lem}\label{conj}
Let $M\in  \mathscr{M}(g, h)$ for commuting $ g,h\in G$.
Assume that $N\in \mathscr{M}(h, g^{-1})$.
For  any $k\in  G$, we have  $S_{M,N}=S_{M\circ k,N\circ k}$.
\end{lem}
\begin{proof}
Note that $M\circ k\in  \mathscr{M}(k^{-1}gk, k^{-1}hk)$ and $N\circ k\in  \mathscr{M}(k^{-1}hk, k^{-1}g^{-1}k)$. Using  Theorem \ref{minvariance}  and Lemma \ref{GM} we have the following computation
\begin{equation*}
\begin{split}
Z_M(v, (g,h), -1/\tau)&=Z_{M\circ k}(v, (k^{-1}gk,k^{-1}hk),-1/\tau)\\
&=\tau^{\wt[v]}\sum_{N\in \mathscr{M}(k^{-1}hk, k^{-1}g^{-1}k) }S_{M\circ k,N}Z_{N}(v,(khk^{-1},k^{-1}g^{-1}k),\tau)\\
&=\tau^{\wt[v]}\sum_{N\in \mathscr{M}(h, g^{-1}) }S_{M\circ k,N\circ k}Z_{N\circ k}(v,(khk^{-1},k^{-1}g^{-1}k),\tau)\\
&=\tau^{\wt[v]}\sum_{N\in \mathscr{M}(h, g^{-1}) }S_{M\circ k,N\circ k}Z_{N}(v,(h,g^{-1}),\tau).
\end{split}
\end{equation*}
Comparing this result with
$$Z_M(v, (g,h), -1/\tau=\tau^{\wt[v]}\sum_{N\in \mathscr{M}(h, g^{-1}) }S_{M,N} Z_{N}(v,(h,g^{-1}),\tau)$$
 and using the fact that
$\{Z_{N}(v,(h,g^{-1}),\tau)  \mid  N\in \mathscr{M}(h, g^{-1})\}$ are linearly independent functions on
$V\times \H$  \cite{DLM5} gives the identity
 $S_{M,N}=S_{M\circ k,N\circ k}$.
\end{proof}

The following result essentially gives the $S$-matrix of $V^G.$
\begin{thm}\label{MVG}  Let $M$ be an irreducible $g$-twisted $V$-module.
For $v\in  V^G$,  $\lambda\in \Lambda_{G_M}$, we have
$$
Z_{M_\l}(v, \tau)=\frac{1}{|G_M|}\sum_{h\in G_M}Z_M(v, (g,h),\tau)\overline{\lambda(\overline{h})},
$$
$$Z_{M_\l}(v, -1/\tau)=\frac{\tau^{\wt[v]}}{|G_M|}\sum_{h\in G_M}\sum_{N\in \mathscr{M}(h,g^{-1})}S_{M,N}\sum_{\mu\in\Lambda_{G_N}}\mu(\overline{g^{-1}})Z_{N_{\mu}}(v,\tau)\overline{\lambda(\overline{h})}$$
where $\bar h$ is the  element in twisted group algebra $\C^{\alpha_M}[G_M]$ corresponding to $h\in G_M$ and $\bar {x}$ for a complex number $x$ is the complex conjugate.
\end{thm}
\begin{proof}
Recall $M=\oplus_{\lambda\in \Lambda_{G_M}}W_{\l}\otimes M_{\lambda}$.
We have
\begin{equation*}
\begin{split}
\frac{1}{|G_M|}\sum_{h\in G_M}Z_M(v, (g,h),\tau)\overline{\lambda(\overline{h})}
=&\frac{1}{|G_M|}\sum_{h\in G_M}\tr_M\overline{h} o(v)q^{L(0)-\frac{c}{24}}\overline{\lambda(\overline{h})}\\
=&\frac{1}{|G_M|}\sum_{h\in G_M}\sum_{\mu\in \Lambda_{G_M}}(\tr_{W_{\mu}}\overline{h} )(\tr_{M_{\mu}}o(v))q^{L(0)-\frac{c}{24}}\overline{\lambda(\overline{h})}\\
=&\frac{1}{|G_M|}\sum_{h\in G_M}\sum_{\mu\in \Lambda_{G_M}}
\mu(\overline{h} )\overline{\lambda(\overline{h})}Z_{M_{\mu}}(v,\tau)
\end{split}
\end{equation*}
Using the the orthogonality property of the irreducible characters of $\C^{\a_M}[G_M]$ gives the first equality.

For the second equality we use Theorem  \ref{minvariance} and the first equation to see that
\begin{equation*}
\begin{split}
Z_{M_\l}(v, -1/\tau)=&\frac{\tau^{\wt[v]}}{|G_M|}\sum_{h\in G_M}\sum_{N\in \mathscr{M}(h,g^{-1})}S_{M,N}Z_N(v,(h,g^{-1}),\tau)\overline{\lambda(\overline{h})}\\
                                =&\frac{\tau^{\wt[v]}}{|G_M|}\sum_{h\in G_M}\sum_{N\in \mathscr{M}(h,g^{-1})}S_{M,N}\sum_{\mu\in\Lambda_{G_N}}\mu(\overline{g^{-1}})\overline{{\lambda(\overline{h})}}Z_{N_{\mu}}(v,\tau),
\end{split}
\end{equation*}
as desired.
\end{proof}

From \cite{Z} and Theorem \ref{mthm1}, $\{Z_{M^j_{\l}}(v,\tau)|j\in J, \l\in \Lambda_{G_{M^j}}\}$  are linearly independent vectors in the conformal block
of $V^G.$  According to the orbifold theory conjecture, these vectors are expected  to form a basis of the conformal block. Note that each $Z_M(v,(g,h),\tau)$ for $M\in  \mathscr{M}(g, h)$
with $g,h\in G$ commuting is also a vector in the conformal block of $V^G.$
The first equation in Theorem \ref{MVG} in fact implies the subspace of the conformal block spanned  $Z_{M^j_{\l}}(v,\tau)$  is equal to the the subspace
spanned by $Z_{M^j}(v, (g_j,h),\tau)$ for $j\in J$ where $M^j$ is a $g_j$-twisted $V$-module.

Taking  $M=M^i$ in Theorem \ref{MVG} we can find the  $S_{M^i_\l,M^j_{\mu}}$ after the identification of $N$ appearing in $Z_{M_\l}(v, -1/\tau)$ with $M^j.$  To give an explicit formula
for $S_{M^i_\l,M^j_{\mu}}$ we introduce some notations. Let $C_{i,j}$ be an least subset of $G$ such that
$$\{M^j\circ k |k\in C_{i,j}\}=O_j\cap (\cup_{h\in G_{M^i}} \mathscr{M}(h, g_i^{-1})).$$
 Clearly, the choice of $C_{i,j}$ is not canonical and  $C_{i,j}$ could be an empty set.
 Note that $V\circ h\cong V$ for all $h\in G.$ So $V$ itself is an orbit with $G_V=G.$ We assume that $0\in J$ and $M^0=V.$ In this case $C_{0,j}$ consists of $k\in G$ such that
$M^j\circ k$  gives all the modules in orbit $O_j.$

Here is an explicit expression of the $S$-matrix of $V^G.$
\begin{thm}\label{ZVG}
Let $i,j\in J$ and $\lambda\in\Lambda_{G_{M^i}}$ and $\mu\in \Lambda_{G_{M^j}}.$
Then
$$S_{M^i_\l,M^j_\mu}=\frac{1}{|G_{M^i}|}\sum_{k\in C_{i,j}}S_{M^i,M^j\circ k}\overline{\lambda(\overline{k^{-1}g_jk})}\mu(\overline{kg_i^{-1}k^{-1}})$$
if $C_{i,j}$ is not empty, and $S_{M^i_\l,M^j_\mu}=0$ otherwise.
Moreover, if $i=0$ we have a simple formula:
$$S_{M^0_\l,M^j_\mu}=\frac{1}{|G_{M^j}|}S_{M^0,M^j}\lambda(g_j^{-1})\dim W_{\mu}$$
for all $j.$
\end{thm}
\begin{proof}
We first note that  $\cup_{h\in G_{M^i}} \mathscr{M}(h, g_i^{-1})=\{M^j\circ k |k\in C_{i,j}, j\in J\}.$   Now in Theorem  \ref{MVG}, replacing $N, h$ by $M^j\circ k, k^{-1}g_jk,$ respectively gives
$$Z_{M^i_\l}(v, -1/\tau)\\
=\frac{\tau^{\wt[v]}}{|G_{M^i}|}\sum_{j\in J}\sum_{k\in C_{i,j}}\sum_{\mu\in \Lambda_{G_{M^j\circ k}}}S_{M^i,M^j\circ k}\overline{{\lambda(\overline{k^{-1}g_jk})}}{\mu(\overline {g_i^{-1}})}Z_{(M^j\circ k)_{\mu}}(v,\tau).$$
Let $\mu\in \Lambda_{G_{M^j}}$ and $k\in C_{i,j}.$ Recall from Section 3 that $\mu\circ k\in  \Lambda_{G_{M^j\circ k}}$ such that $(\mu\circ k)(\overline{k^{-1}ak})=\mu(\overline{a})$ for $a\in G_{M^j}.$
In particular, $(\mu\circ k)(\overline{ g_i^{-1}})
= \mu(\overline{kg_i^{-1}k^{-1}}).$
Since $\Lambda_{G_{M^j\circ k}}= \Lambda_{G_{M^j}}\circ k$ and $Z_{(M^j\circ k)_{\mu\circ k}}(v,\tau)=Z_{M^j_{\mu}}(v,\tau)$ for $\mu\in  \Lambda_{G_{M^j}}$
we have
$$Z_{M^i_\l}(v, -1/\tau)\\
=\frac{\tau^{\wt[v]}}{|G_{M^i}|}\sum_{j\in J}\sum_{k\in C_{i,j}}\sum_{\mu\in \Lambda_{G_{M^j}}}S_{M^i,M^j\circ k}\overline{{\lambda(\overline{k^{-1}g_jk})}}\mu(\overline {kg_i^{-1}k^{-1}})Z_{M^j_{\mu}}(v,\tau)$$
or equivalently,
$$S_{M^i_\l,M^j_\mu}=\frac{1}{|G_{M^i}|}\sum_{k\in C_{i,j}}S_{M^i,M^j\circ k}\overline{\lambda(\overline{k^{-1}g_jk})}\mu(\overline{kg_i^{-1}k^{-1}}).$$

If $i=0$ then $g_0=1,$
$|C_{0,j}|=\frac{|G|}{|G_{M^j}|},$
$$\overline{\lambda(\overline{k^{-1}g_jk})}=\overline{\lambda(k^{-1}g_jk)}=\overline{\lambda(g_j)}=\lambda(g_j^{-1}),$$
$\mu(\overline {kg_0^{-1}k^{-1}})=\dim W_{\mu}.$
Using Lemma \ref{conj}  yields
$$S_{M^0,M^j\circ k}=S_{M^0\circ k,M^j\circ k}=S_{M^0,M^j}$$
and
$$S_{M^0_\l,M^j_\mu}=\frac{1}{|G_{M^j}|}S_{M^0,M^j}\lambda(g_j^{-1})\dim W_{\mu}.$$
The proof is complete.
\end{proof}

\begin{coro}\label{coro1} Assume that $G$ is abelian.  Let $i,j, \lambda, \mu$ be as before. Then
$$S_{M^i_\l,M^j_\mu}=\frac{1}{|G_{M^i}|}\sum_{k\in C_{i,j}}S_{M^i,M^j\circ k}\overline{\lambda(\overline{g_j})}\mu(\overline{g_i^{-1}})$$
if $C_{i,j}$ is not empty, and $S_{M^i_\l,M^j_\mu}=0$ otherwise.
\end{coro}
\begin{proof}
As $G$ is abelian, $k^{-1}g_jk=g_j$ and $kg_i^{-1}k^{-1}=g_i^{-1}.$  The result follows from Theorem  \ref{ZVG} immediately.
\end{proof}

Next we assume that $G$ is holomorphic. Then for each $g\in G$ there is a unique irreducible $g$-twisted $V$-module $V(g)$ \cite{DLM5} and $G_{V(g)}=C_G(g).$
In this case $\{g_i|i\in J\}$ is a set of conjugacy class representatives, $M^i=V(g_i),$  $V(g)\circ k=V(k^{-1}gk)$ and $|C_{i,j}|$ is exactly the cardinality of the
intersection of the conjugate class of $g_j$ in $G$ with $C_G(g_i).$ If  we also assume that $G$ is abelian, then  $J=G,$ then $C_{g,h}=\{h\}$ and $C_G(g)=G.$

\begin{coro}\label{coro2} Let $V$ be holomorphic and $i,j, \lambda, \mu$ be as before. We have
$$S_{V(g_i)_\l,V(g_j)_\mu}=\frac{1}{|C_G(g_i)|}\sum_{k\in C_{i,j}}S_{V(g_i),V(k^{-1}g_jk)}\overline{\lambda(\overline{k^{-1}g_jk})}\mu(\overline{kg_i^{-1}k^{-1}}).$$
If $G$ is abelian the formula simplifies to
$$S_{V(g)_\l,V(h)_\mu}=\frac{1}{|G|}S_{V(g),V(h)}\overline{\lambda(\overline{h})}\mu(\overline{g^{-1}}).$$
\end{coro}

\section{Permutation orbifolds}

Our next example comes from the cyclic permutation orbifolds from \cite{BDM}.
Let $V$ be a vertex operator algebra satisfying conditions (V1)-(V3) with $G$ being automorphism group of $V$ .
For a fixed  prime positive integer $k$,
consider the tensor product vertex operator algebra $U=V^{ \otimes k}$ \cite{FHL}.
Any element $g$ of the permutation group $S_k$ acts on $U$ in the obvious way.  Take $g=(1,2,\ldots,k)$ to be be a $k$-cycle permutation of the vertex operator algebra .
Our goal is to determine the $S$ matrix of $(V^{\otimes k})^G$ where $G$ is the cyclic group generated
by $g.$ According to Corollary \ref{coro1} we only need to know the $S$-matrix in Theorem \ref{minvariance}. In the permutation orbifolds, we can use the the action of
the modular group on the conformal block  of $V$ to give an explicit formula of $S$-matrix of $(V^{\otimes k})^G$ due to the connection between $V$-modules and $g$-twisted
$(V^{\otimes k})^G$-modules \cite{BDM}.

Suppose that $M^0, M^1,\ldots, M^p$ are the irreducible $V$-modules up to isomorphism.
Then $V^{ \otimes k}$ is ratioanl and all irreducible modules are
$$\{M^{i_1}\otimes M^{i_2}\cdots \otimes M^{i_k}\mid i_1,\ldots,
 i_k \in\{0,\ldots ,p\} \}.$$
It is known from \cite{DLM5} that the number of irreducible $g$-stable $V^{ \otimes k}$-modules is equal to the number of of irreducible $V$ -modules and
the number of irreducible  $g$-twisted $V^{ \otimes k}$-modules up to isomorphism is equal to the number of  irreducible
$V$ -modules.
It is proved in \cite{BDM} that there is equivalent functor $T_g$ from the category of  $V$ -module category to the $g$-twisted  $V^{ \otimes k}$-module category.
In particular, $T_g^k(M^0),T_g^k(M^1),\ldots, T_g^k(M^p)$ are the irreducible  $g$-twisted $V^{ \otimes k}$-modules.
There is $\frac{1}{k}\Z_+$-gradation on $T_g^k(M^i)$ such that
$T_g^k(M^i)=\oplus _{n\geq 0 }T_g^k(M^i)(\frac{n}{k})$ and $T_g^k(M^i)(\frac{n}{k}) \cong M^i(n)$ as vector space, and
$Y_g(v,z)=\sum_{m\in \frac{1}{k}\Z}v_mz^{-m-1}$ for $v\in V^{\otimes k}.$

For $v\in V$ denote by $v^j\in V^{ \otimes k}$ the vector whose $j$-th tensor factor is $v$ and whose other tensor factors are 1.
Then $gv^j=v^{j+1}$ for $j=1,\ldots, k$,  where $k+1$ is understood to be 1.
Let $W$ be a $g$-twisted $V^{ \otimes k}$-module, and let $\eta=e^{-2\pi i /k}$.
Then
\[
Y_g(v^{j+1},z)=\lim_{z^{1/k}\to \eta^{-j}z^{1/k} }Y_g(v^1,z).
\]
Since $V^{ \otimes k}$ is generated by $v^j$ for $v\in V$ and $j=1,2,\ldots, k$, the vertex operaotrs $Y_g(v^1,z)$ for $v\in V $
determined all the vertex operators $Y_g(u,z)$ on $W$ for any $u\in V^{ \otimes k}$.

In $(\End V)[[z^{1/k}, z^{-1/k}]],$  define $\Delta_k(z)=exp\big(\sum_{j\in \Z_+}a_j z^{-j/k}L(j)\big)k^{L(0)}z^{(1/k-1)L(0)}$ (see \cite{BDM} for detials).
Let $M$ be a  $V$-module. The action of $V^{ \otimes k}$ on $T_g(M)$  is uniquely determined by
$Y_g(v^1,z)=Y(\Delta_k(z)v, z^{1/k})$ for $v\in V.$
If $v\in V_n$ is a highest weight vector then $\Delta_k(z)v=k^{-n}z^{(1/k-1)n}v$ and $Y_g(v^1,z)=k^{-n}z^{(1/k-1)n}Y(v, z^{1/k})$.

For the Virasoro vector  $\omega$  of $V$, $\Delta_k(z)\omega=\frac{z^{2(1/k-1)}}{k^2}(\omega+\frac{(k^2-1)c}{24}z^{-2/k})$.
We write $Y_{g}(\bar{\omega}, z)=\sum_{n \in \mathbb{Z}} L_{g}(n) z^{-n-2},$
where $\bar{v}=\sum_{j=1}^{k}v^{j}$ for any $v\in V.$
We have $Y_{g}(\bar{\omega}, z)=\sum_{i=0}^{k-1} \lim _{z^{1 / k} \mapsto \eta^{-i} z^{1 / k}} Y_{g}\left(\omega^{1}, z\right).$
It follows that $L_{g}(0)=\frac{1}{k} L(0)+\frac{\left(k^{2}-1\right) c}{24 k}.$

\begin{lem}\label{two}
Let $v\in V$, then $o(\bar{v})=ko(v^1)$.
Moreover, if $v\in V_n$ is a highest weight vector for the Virasoro algebra, then
$o(\bar{v})=k^{-n+1}o(v).$
\end{lem}

\begin{proof}
It is good enough to show the result for homogeneous $v.$ Note that
$$
Y_g(g v^{1},z)=\lim_{z^{1/k}\to e^{\frac{2\pi i}{k}}z^{1/k} }Y_g(v^1,z)
= \sum_{m\in\Z}v^1_{\frac{m}{k}}(e^{\frac{2\pi i}{k}}z^{\frac{1}{k}})^{-m-k}
= \sum_{m\in\Z}v^1_{\frac{m}{k}}e^{\frac{-2m\pi i}{k}}z^{\frac{-m-k}{k}}.
$$
Comparing the coefficients of $z^{-\wt v}$,
we have $o(v^2)=o(gv^1)=o(v^1)$ and $o(\bar{v})=ko(v^1)$.
If $v\in V_n$ is a highest weight vector, we have
$Y_g( v^{1},z)=k^{-n}z^{(1/k-1)n}Y(v, z^{1/k})$ and
$o(v^1)=k^{-n}o(v)$. The result follows.
\end{proof}

Here is a general result  which will be used later.
\begin{lem}\label{general}
Let $\sigma$ be an automorphism  of $V$ of order $T$ and  $M=\sum_{n\geq 0}M_{\frac{n}{T}+\lambda}$  a $\sigma$-twisted $V$-module where $\lambda$ is the weight of $M.$
Then $M\circ \sigma^r$ and $M$ are isomorphic for any integer $r$ and $Z_M(v, (\sigma, \sigma^r), \tau)=e^{2\pi ir(\frac{c}{24}-\lambda)}Z_M(v, \tau+r).$
\end{lem}
\begin{proof}
We have mentioned already that  for any $\sigma$-twisted module is $\sigma$-stable. Thus any $\sigma$-twisted module is also $\sigma^r$-stable.
By the definition of trace function, we have
\begin{equation*}
\begin{split}
Z_M(v, (\sigma, \sigma^r), \tau)&=\tr_M o(v)\varphi(\sigma^r)q^{L(0)-\frac{c}{24}}\\
& =\sum_{n\geq 0}\tr_{M_{\frac{n}{T}+\lambda }}o(v)e^{\frac{2\pi irn }{T}}q^{\frac{n}{T}+\lambda-\frac{c}{24}}\\
& =e^{2\pi ir(\frac{c }{24}-\lambda)}\sum_{n\geq 0}\tr_{M_{\frac{n}{T}+\lambda }}o(v)(e^{2\pi ir}q)^{L(0)-\frac{c}{24}}\\
& =e^{2\pi ir(\frac{c }{24}-\lambda)}Z_M(v,  \tau+r),
\end{split}
\end{equation*}
as desired.
\end{proof}

We first assume that  $g=(1,2,\ldots,k)$ with  $k$ being a prime. This assumption will make the computations much easier as $\sigma^r$ is also a $k$-cycle for $r=1,...,k-1.$
Recall that $M^0, M^1,\ldots, M^p$ are the irreducible $V$-modules up to isomorphism.
Let $M^j=\oplus _{n\geq 0} M^j_{\lambda_j+n}$, where $j=0,\ldots,p$.
Then the irreducible $g^r$-twisted $V^{\otimes k}$-module are $T_{g^r}^k(M^j)$
 and
 $$T_{g^r}^k(M^j)= \oplus_{n\geq 0} T_{g^r}^k(M^j)_{\frac{\lambda_j}{k}+\frac{(k^2-1)c}{24k}+\frac{n}{k}}$$
 where  $T_{g^r}^k(M^j)_{\frac{\lambda_j}{k}+\frac{(k^2-1)c}{24k}+\frac{n}{k}}=M^j_{\lambda_j+n}$ as vector spaces.

\begin{lem}\label{rs}
Suppose that  $r,s\in \N$ with $r$ positive such that $s\equiv ra\ (\mod k)$ for some integer $a,$ and $v\in V$ is a highest weight vector for the Virasoro algebra,
then
$$Z_{T_{g^r}^k(M^j)}(\bar{v}, (g^r, g^s), \tau)=\e^{2\pi i a (-\frac{\lambda_j}{k}+\frac{c}{24k})}k^{-\wt v+1}Z_{M^j}(v, \frac{\tau+a}{k}).$$
\end{lem}
\begin{proof}
A straightforward calculation using Lemmas \ref{two} and \ref{general} gives
\begin{equation*}
\begin{split}
Z_{T_{g^r}^k(M^j)}(\bar{v}, (g^r, g^s), \tau)
&=e^{2\pi i a (-(\frac{\lambda_j}{k}+\frac{(k^2-1)c}{24k})+\frac{kc}{24}})Z_{T_{g^r}^k(M^j)}(\bar{v},  \tau+a)\\
&=\e^{2\pi i a (-\frac{\lambda_j}{k}+\frac{c}{24k})}\sum_{n\geq 0 } tr_{T_{g^r}^k(M^j)_{\frac{\lambda_j}{k}+\frac{(k^2-1)c}{24k}+\frac{n}{k}}}o(\bar{v})(e^{2\pi i(\tau+a)})^{ (L_g(0)-\frac{kc}{24})}\\
& =\e^{2\pi i a (-\frac{\lambda_j}{k}+\frac{c}{24k})}\sum_{n\geq 0}\tr_{M^j_{n+\lambda_j}}k^{-\wt v+1}o(v)e^{2\pi i(\tau+a)(\frac{L(0)}{k}+\frac{(k^2-1)c}{24k}-\frac{c}{24})}\\
& =\e^{2\pi i a (-\frac{\lambda_j}{k}+\frac{c}{24k})}k^{-\wt v+1}Z_{M^j}(v, \frac{\tau+a}{k}),
\end{split}
\end{equation*}
as expected.
\end{proof}

\begin{lem}\label{l5.4}
Let $0<r,s,a,b<k$ be positive integers such that $s\equiv ra$,$-r\equiv sb$ modulo $k.$
Then
$$S_{T_{g^r}^k(M^j),T^k_{g^s}(M^i)}=e^{2\pi i a (-\frac{\lambda_j}{k}+\frac{c}{24k})-2\pi i b (-\frac{\lambda_i}{k}+\frac{c}{24k})}\sum_{l=0}^{p}S_{M^j,M^i} A_{l,i}^{r,s}, $$
where $A_{l,i}^{r,s}$ is the entry of $\rho(A^{r,s})$  defined in Theorem \ref{minvariance} with $A^{r,s}=\begin{pmatrix} k & -b\\ -a & \frac{1+ab}{k} \end{pmatrix}.$
\end{lem}
\begin{proof}
We first  prove that $A^{r,s}\in \SL_2(\Z).$ Clearly, the determinant of $A$ is $1.$ So we only need to show that $ab+1$ is divizable by $k.$
Since $k$ is a prime, $a,b$ are invertible modulo $k$ and $ab\cong -1$ modulo $k$ which is equivalent to the fact that $\frac{1+ab}{k}$ is an integer.

From  Lemma \ref{rs} we see that
\begin{equation*}
\begin{split}
Z_{T_{g^r}^k(M^j)}(\bar{v}, (g^r, g^s), -1/\tau)
& =\e^{2\pi i a (-\frac{\lambda_j}{k}+\frac{c}{24k})}k^{-\wt v+1}Z_{M^j}(v, \frac{\frac{-1}{\tau}+a}{k})\\
& =\e^{2\pi i a (-\frac{\lambda_j}{k}+\frac{c}{24k})}k^{-\wt v+1}Z_{M^j}(v, -\frac{1}{\frac{\tau k}{1-a\tau}})\\
& =\e^{2\pi i a (-\frac{\lambda_j}{k}+\frac{c}{24k})}k^{-\wt v+1}(\bar{\tau}^{\wt v})\sum_{i=0}^{p}S_{M^j,M^i}Z_{M^i}(v, \bar{\tau})
\end{split}
\end{equation*}
where $ \bar{\tau}=\frac{\tau k}{1-a\tau}$.
Note that $ \bar{\tau}=A^{r,s} \frac{\tau +b}{k}$
and
$$Z_{M^i}(v, \bar{\tau})=(\frac{1-a\tau}{k})^{\wt v}\sum_{l=0}^{p} A_{i,l}^{r.s}Z_{M^l}(v,\frac{\tau +b}{k}).$$
Thus,
\begin{equation*}
\begin{split}
& Z_{T_{g^r}^k(M^j)}(\bar{v}, (g^r, g^s), -1/\tau)\\
& =\e^{2\pi i a (-\frac{\lambda_j}{k}+\frac{c}{24k})}k^{-\wt v+1}(\frac{\tau k}{1-a\tau})^{\wt v}\sum_{i=0}^{p}S_{M^j,M^i}(\frac{1-a\tau}{k})^{\wt v}\sum_{l=0}^{p} A_{i,l}^{r,s}Z_{M^l}(v,\frac{\tau +b}{k})\\
& =\e^{2\pi i a (-\frac{\lambda_j}{k}+\frac{c}{24k})}k^{-\wt v+1}\tau^{\wt v}\sum_{i,l=0}^{p}S_{M^j,M^i} A_{i,l}^{r,s}Z_{M^l}(v,\frac{\tau +b}{k})\\
& =\e^{2\pi i a (-\frac{\lambda_j}{k}+\frac{c}{24k})}k^{-\wt v+1}\tau^{\wt v}\sum_{i,l=0}^{p}S_{M^j,M^l} A_{l,i}^{r,s}Z_{M^i}(v,\frac{\tau +b}{k})
\end{split}
\end{equation*}
On the other hand,
\begin{equation*}
\begin{split}
& Z_{T_{g^r}^k(M^j)}(\bar{v}, (g^r, g^s), -1/\tau)\\
& =\tau^{\wt v}\sum_{i=0}^{p}S_{T_{g^r}^k(M^j),T^k_{g^s}(M^i)}Z_{T^k_{g^s}(M^i)}(\bar{v}, (g^s,g^{-r}),\tau)\\
& =\tau^{\wt v}\sum_{i=0}^{p}S_{T_{g^r}^k(M^j),T^k_{g^s}(M^i)}\e^{2\pi i b (-\frac{\lambda_i}{k}+\frac{c}{24k})}k^{-\wt v+1}Z_{M^i}(v, \frac{\tau+b}{k}).
\end{split}
\end{equation*}
Comparing the right sides of these equations,  we have
\[
S_{T_{g^r}^k(M^j),T^k_{g^s}(M^i)}=e^{2\pi i a (-\frac{\lambda_j}{k}+\frac{c}{24k})-2\pi i b (-\frac{\lambda_i}{k}+\frac{c}{24k})}\sum_{l=0}^{p}S_{M^j,M^i} A_{l,i}^{r,s}.
\]
We should point out that the matrix $A^{r,s}$ depends on $r,s$ only.
\end{proof}

The following result is easy:
\begin{lem}\label{character}
Let $M$ be an irreducible $V$-module. For $g\in G$, $v\in V$ and $0<r<k$ then

(1) $Z_{M^{\otimes k}}(\1, (1,g^r),\tau)= \chi _{M}(k \tau)$

(2) $Z_{M^{\otimes k}}(\bar{v}, (1,g^r),\tau)=kZ_{M}(v,k \tau).$
\end{lem}

\begin{proof}
(1) follows from (2) by noting that $\bar \1=k(\1\otimes \cdots \otimes \1).$ So we only need to prove (2). Let $\{w^\alpha|\alpha\in A\}$ be a basis of $M$ such that each
$w^{\alpha}$ is homogeneous and a generalized eigenvector of $o(v).$ Then $\{w^{\alpha_1}\otimes \cdots \otimes w^{\alpha_k}|\alpha_i\in A\}$ is a basis of $M^{\otimes k}$
and the action of $G$ on $M^{\otimes k}$ preserves this basis.  Note that each $w^{\alpha_1}\otimes \cdots \otimes w^{\alpha_k}$ is a generalized eigenvector of $o(\bar v)$ with eigenvalue $\sum_{i=1}^k\lambda_i$ where $\lambda_i$ is the eigenvalue of $w^{\alpha_i}.$ Moreover, for any $h\in G,$ $h(
w^{\alpha_1}\otimes \cdots \otimes w^{\alpha_k})$ is also a generalized eigenvector of $o(\bar v)$ with the same eigenvalue $\sum_{i=1}^k\lambda_i.$
 If there are at least two different $\alpha_i$ in $w^{\alpha_1}\otimes \cdots \otimes w^{\alpha_k},$ then $\sum_{s=0}^{k-1}\C g^sw^{\alpha_1}\otimes \cdots \otimes w^{\alpha_k}$
 is a $k$-dimensional subspace of $M^{\otimes k}$ and has a basis consisting of generalized eigenvectors of $o(\bar v)g^r$ with eigenvalues $(\sum_{i=1}^k\lambda_i)e^{\frac{2\pi i s}{k}}$
  for $s=0,...,k-1.$ Thus the trace contribution from this subspace is $0.$ So we only need to compute the trace contribution from vectors $w^{\alpha}\otimes \cdots w^{\alpha}$ for
  $\alpha\in A.$

 Let $\lambda_{\alpha}$ be the eigenvalue of $w^{\alpha}.$ Note  that $\omega_{V^{\otimes k}}=\bar \omega.$ Denote the corresponding $L(0)$ by $L_{V^{\otimes k}}(0).$
  Then $w^{\alpha}\otimes \cdots w^{\alpha}$ is a generalized eigenvector of $o(\bar v)g^rq^{L_{V^{\otimes k}}(0)-\frac{kc}{24}}$ with eigenvalue
   $k\lambda_{\alpha} q^{k(\wt w^{\alpha}-\frac{c}{24})}.$  Thus we have
\begin{equation*}
\begin{split}
Z_{(M)^{\otimes k}}(\bar{v}, (1,g^r),\tau)&=\tr_{M^{\otimes k}}o(\bar v)g^rq^{L_{V^{\otimes k}}(0)-kc/24}\\
&=\sum_{\alpha\in A}k\lambda_{\alpha} q^{k(\wt w^{\alpha}-\frac{c}{24})}\\
& =kZ_{M}(v,k \tau).
\end{split}
\end{equation*}
The proof is finished.
\end{proof}

\begin{lem}\label{case2}
Let $g\in G$, $M^i,$  $r$ be as before, $i=0,\ldots, p$. Then
\[
S_{(M^i)^{\otimes k}, T_{g^r}^k(M^j)}=S_{M^i,M^j}=S_{T_{g^r}^k(M^i),(M^j)^{\otimes k}}.
\]
\end{lem}
\begin{proof}
Using Lemma \ref{character} yields
\[
Z_{(M^i)^{\otimes K}}(\bar{v}, (1,g^r),-\frac{1}{\tau})=k Z_{M^i}(v,\frac{-1}{\frac{\tau}{k}})=k (\frac{\tau}{k})^{\wt v}\sum_{j=0}^{p}S_{M^i,M^j}Z_{M^j}(v,\frac{\tau}{k} ).
\]
By Lemma \ref{rs} with $s=a=0,$
 \[
Z_{T_{g^r}^k(M^j)}(\bar{v}, (g^r,1)),\tau)=k^{-\wt v+1}Z_{M^j}(v,\frac{\tau}{k}).
 \]
 So
 \[
Z_{(M^i)^{\otimes K}}(\bar{v}, (1,g^r),-\frac{1}{\tau}) =\tau^{\wt v}\sum_{j=0}^{p}S_{M^i,M^j}Z_{T_{g^r}^k(M^j)}(\bar{v}, (g^r,1)),\tau)
\]
On the other hand,
\[
Z_{(M^i)^{\otimes K}}(v, (1,g^r),-\frac{1}{\tau})=\tau^{\wt v}\sum_{j=0}^p S_{(M^i)^{\otimes k}, T_{g^r}^k(M^j)}Z_{T_{g^r}^k(M^j)}(v,(g^r,1),\tau).
\]
Comparing the right sides of these equations gives $S_{(M^i)^{\otimes k}, T_{g^r}^k(M^j)}=S_{M^i,M^j}$ for all $i,j.$

Similarly,
 \begin{equation*}
 \begin{split}
Z_{T_{g^r}^k(M^i)}(\bar{v}, (g^r,1),-\frac{1}{\tau})& =k^{-\wt v+1}Z_{M^i}(v,\frac{-1}{k\tau})\\
&=k^{-\wt v+1}(k\tau)^{\wt v}\sum_{j=0}^pS_{M^i,M^j}Z_{M^j}(v,k\tau)\\
&=\tau^{\wt v}\sum_{j=0}^pS_{M^i,M^j}Z_{(M^j)^{\otimes k}}(v,\tau)
\end{split}
\end{equation*}
and $S_{M^i,M^j}=S_{T_{g^r}^k(M^i),(M^j)^{\otimes k}}.$
\end{proof}

 We are now ready to compute the $S$-matrix for $(V^{\otimes k})^G.$ First we give a complete list of irreducible $(V^{\otimes k})^G$-modules following \cite{DRX}.
 Let $i_1,...,i_k\in [0,p] $ where  $[0,p]=\{0,...,p\}.$ Set $M^{i_1,...,i_k}=M^{i_1}\otimes \cdots \otimes M^{i_k}.$  Then $M^{i_1,...,i_k}$ is an irreducible
  $(V^{\otimes k})^G$-module if the cardinality $\{i_1,...,i_k\}$ is greater than  $1.$ Moreover, $M^{i_1,...,i_k}\circ g^r=M^{i_{1+r},\cdots, i_{k+r}}$ for $r=0,...,k-1$ where
  the sub-index $j+r$ is understood to be modulo $k.$ The
  $M^{i,...,i}$ is a  direct sum  of irreducible $(V^{\otimes k})^G$-modules    $(M^{i,...,i})^j$ for   $j=0,...,k-1$ where
  $$(M^{i,...,i})^j=\{w\in  M^{i,...,i}|gw=e^{\frac{-2\pi i j}{k}}w\}=\{\sum_{s=0}^{k-1}e^{\frac{2\pi i sj}{k}}g^sw|  w\in  M^{i,...,i}\}$$
  and $g$ acts on  $M^{i,...,i}$ in an obvious way.
  We have already mentioned that $G$ acts on each $T_{g^r}^k(M^i).$ So $T_{g^r}^k(M^i)=\oplus_{s=0}^{k-1} (T_{g^r}^k(M^i))^s$
  is a direct sum of irreducible  $(V^{\otimes k})^G$-modules such that $g$ acts on  $(T_{g^r}^k(M^i))^s$ as    $e^{\frac{-2\pi i s}{k}} .$
 Note that $\Irr(G)=\{\lambda_s|s=0,...,k-1\}$ where $\lambda_s(g)=
 e^{\frac{-2\pi i s}{k}}.$

  Note that the symmetric group $S_k$ acts on $[0,p]^k\setminus \{(i,\cdots,i)|i\in[0,p]\}$ in an obvious way. Let $I$ be a subset of
  $[0,p]^k\setminus \{(i,\cdots,i)|i\in[0,p]\}$ consisting of the orbit representatives under the action of $G.$

\begin{prop} The irreducible  $(V^{\otimes k})^G$-modules consist of
$$\{M^{i_1,...,i_k}| (i_1,...,i_k)\in I\}$$
and
$$\{ (M^{i,\cdots,i})^a ,  (T_{g^r}^k(M^i))^a|0\leq i\leq p, 0\leq a\leq k-1, 0<r<k\}.$$
\end{prop}
 The following theorem gives explicit expressions of  the entries of the  $S$-matrix of $(V^{\otimes k})^G$ by noting that $S$-matrix is symmetric \cite{H}.

\begin{thm} (1) Let $i,j=0,...,p,$  $0<r,s<k,$ $0\leq a,b<k.$ Then
  $$S_{(T_{g^r}^k(M^i))^a, N}=\frac{1}{k}
  \left\{\begin{array}{ll}
S_{T_{g^r}^k(M^i), T_{g^s}^k(M^j) }e^{\frac{2\pi i (sa+rb)}{k}}  & {\rm if}\  N=(T_{g^s}^k(M^j))^b \\
S_{M^i,M^j}e^{\frac{2\pi i rb}{k}}
 & {\rm if}  N=  (M^{j,\ldots, j})^b \\
0  & {\rm otherwise}
 \end{array}\right.$$
where $S_{T_{g^r}^k(M^i), T_{g^s}^k(M^j) } $ is given in Lemma \ref{l5.4} and  $S_{M^i,M^j} $ is the entry of $S$-matrix of $V$.

(2) Let $(i_1,...,i_k), (t_1,...,t_k)\in I.$ and $i\in[0,p],$ $0\leq b<k.$  Then
 $$S_{M^{i_1,...,i_k}, N}=\left\{\begin{array}{ll}\sum_{r=0}^{k-1}\prod_{j=1}^k S_{M^{i_j},M^{t_{j+r}}}& {\rm if}\  N=M^{t_1,...,t_k} \\  \prod_{j=1}^k S_{M^{i_j},M^i} & {\rm if} \ N=  (M^{j,\ldots, j})^b
 \end{array}\right.$$

 (3) Let $i,j\in [0,p]$ and $a,b\in[0,k-1].$ Then
 $$S_{(M^{i,...,i})^a,(M^{j,...,j})^b } =\frac{1}{k}S_{M^i,M^j}^k$$
  \end{thm}
  \begin{proof}

 (1) In Corollary \ref{coro1}, let $g_i=g^r,$ $g_j=g^{s'}$ with $s'=s$ or $1$  and $\lambda=\lambda_a, \mu=\lambda_b,$ then $\overline{\lambda_a(g^s) }\lambda_b(g^{-r})=e^{\frac{2\pi i (s'a+rb)}{k}}.$ The result follows.

 (2)  Let $i_1,...,i_k\in [0,p],$  $v_1,...,v_k\in V.$ Then
$$Z_{M^{i_1,...,i_k}}(v_1\otimes\cdots \otimes v_k,\tau)=Z_{M^{i_1}}(v_1,\tau)\cdots Z_{M^{i_k}}(v_k,\tau)$$
So we have
 \begin{equation*}
 \begin{split}
 Z_{M^{i_1,...,i_k}}(v_1\otimes\cdots \otimes v_k-\frac{1}{\tau})
 &=\tau^{\wt v}\prod_{j=1}^k\left(\sum_{t_j=0}^pS_{M^{i_j},M^{t_j}}Z_{M^{t_j}}(v_j,\tau)\right)\\
 &=\tau^{\wt v}\sum_{(t_1,...,t_k)\in[0,p]^k} \prod_{j=1}^k S_{M^{i_j},M^{t_j}}Z_{M^{t_1,...,t_k}}(v_1\otimes\cdots\otimes v_k,\tau).
  \end{split}
\end{equation*}
So $S_{M^{i_1,...,i_k},M^{t_1,...,t_k}}=\prod_{j=1}^k S_{M^{i_j},M^{t_j}}$ for irreducible $V^{\otimes k}$-modules.

Let $(i_1,...,i_k), (t_1,...,t_k)\in I.$
Using Corollary \ref{coro1} with $M^i=M^{i_1,...,i_k}, M^j=M^{t_1,...,t_k}, g_i=g_j=1, C_{i,j}=G$ and $\lambda=\mu=1$
gives
 $$S_{M^{i_1,...,i_k},M^{t_1,...,t_k}}=\sum_{r=0}^{k-1}\prod_{j=1}^k S_{M^{i_j},M^{t_{j+r}}}.$$
 Now let $i\in[0,p]$ and $b=0,...,k-1.$ Using Corollary \ref{coro1} again with $M^i=M^{i_1,...,i_k}, M^j=M^{i,\ldots, i},  g_i=g_j=1, \lambda=1, \mu=\lambda_b$ gives
   $$S_{M^{i_1,...,i_k},(M^{i,...,i})^b}=\prod_{j=1}^k S_{M^{i_j},M^i}.$$

(3) can be proved similarly.
\end{proof}

 In principle one can compute the $S$-matrix of $V^{\otimes k}$ for any $k.$ The idea is clear but the computation will be more complicated as $g^r$ could be a product of
 several disjoint cycles.

\end{document}